\documentclass[]{amsart}
\usepackage{amsthm}
\usepackage{amstext}
\usepackage{amssymb}
\usepackage{xargs}[2008/03/08]
\usepackage[all]{xy}

\makeatletter
\numberwithin{equation}{section}
\numberwithin{figure}{section}
\theoremstyle{plain}
\newtheorem{thm}{\protect\theoremname}[section]
  \theoremstyle{plain}
  \newtheorem{conjecture}[thm]{\protect\conjecturename}
  \theoremstyle{definition}
  \newtheorem{defn}[thm]{\protect\definitionname}
  \theoremstyle{plain}
  \newtheorem{lem}[thm]{\protect\lemmaname}
  \theoremstyle{remark}
  \newtheorem{rem}[thm]{\protect\remarkname}
  \theoremstyle{plain}
  \newtheorem{prop}[thm]{\protect\propositionname}
  \theoremstyle{plain}
  \newtheorem{cor}[thm]{\protect\corollaryname}
  \theoremstyle{definition}
  \newtheorem{example}[thm]{\protect\examplename}
  \theoremstyle{definition}
  \newtheorem{problem}[thm]{\protect\problemname}

\makeatother

  \providecommand{\conjecturename}{Conjecture}
  \providecommand{\corollaryname}{Corollary}
  \providecommand{\definitionname}{Definition}
  \providecommand{\examplename}{Example}
  \providecommand{\lemmaname}{Lemma}
  \providecommand{\problemname}{Problem}
  \providecommand{\propositionname}{Proposition}
  \providecommand{\remarkname}{Remark}
\providecommand{\theoremname}{Theorem}

\begin{document}

\title{Toward motivic integration over wild Deligne-Mumford stacks}

\author{Takehiko Yasuda}

\address{Department of Mathematics, Graduate School of Science, Osaka University,
Toyonaka, Osaka 560-0043, Japan}

\email{takehikoyasuda@math.sci.osaka-u.ac.jp}

\dedicatory{To Professor Yujiro Kawamata}

\thanks{This work was supported by Grants-in-Aid for Scientific Research
(22740020).}
\begin{abstract}
We discuss how the motivic integration will be generalized to wild
Deligne-Mumford stacks, that is, stabilizers may have order divisible
by the characteristic of the base or residue field. We pose several
conjectures on this topic. We also present some possible applications
concerning stringy invariants, resolution of singularities, and weighted
counts of extensions of local fields.
\end{abstract}

\maketitle
\global\long\def\AA{\mathbb{A}}
\global\long\def\PP{\mathbb{P}}
\global\long\def\NN{\mathbb{N}}
\global\long\def\GG{\mathbb{G}}
\global\long\def\ZZ{\mathbb{Z}}
\global\long\def\QQ{\mathbb{Q}}
\global\long\def\CC{\mathbb{C}}
\global\long\def\FF{\mathbb{F}}
\global\long\def\LL{\mathbb{L}}
\global\long\def\RR{\mathbb{R}}

\global\long\def\bx{\mathbf{x}}
\global\long\def\bf{\mathbf{f}}

\global\long\def\cN{\mathcal{N}}
\global\long\def\cW{\mathcal{W}}
\global\long\def\cY{\mathcal{Y}}
\global\long\def\cM{\mathcal{M}}
\global\long\def\cF{\mathcal{F}}
\global\long\def\cX{\mathcal{X}}
\global\long\def\cE{\mathcal{E}}
\global\long\def\cJ{\mathcal{J}}
\global\long\def\cO{\mathcal{O}}
\global\long\def\cD{\mathcal{D}}

\global\long\def\fs{\mathfrak{s}}
\global\long\def\fp{\mathfrak{p}}
\global\long\def\fm{\mathfrak{m}}

\global\long\def\Spec{\mathrm{Spec}\,}
\global\long\def\AS{\mathrm{AS}}
\global\long\def\mAS{\mathrm{mAS}}
\global\long\def\RP{\mathrm{RP}}
\global\long\def\rj{\mathrm{rj}}
\global\long\def\bRP{\mathbf{RP}}
\global\long\def\ord{\mathrm{ord}\,}
\global\long\def\ordfnc{\mathrm{ord}}
\global\long\def\GCov#1{G\text{-}\mathrm{Cov}(#1)}
\global\long\def\HCov#1{H\text{-}\mathrm{Cov}(#1)}
\global\long\def\GCovrep#1{G\text{-}\mathrm{Cov}^{\mathrm{rep}}(#1)}
\global\long\def\ArbCov#1#2{#1\text{-}\mathrm{Cov}(#2)}
\global\long\def\GCovPt#1{G\text{-}\mathrm{Cov}^{\mathrm{pt}}(#1)}
\global\long\def\bGCovrep#1{\mathrm{G}\text{-}\mathbf{Cov}^{\mathrm{rep}}(#1)}
\global\long\def\Cov#1{\mathrm{Cov}(#1)}
\global\long\def\geoeq{\sim_{\mathrm{geo}}}
\newcommandx\Hom[4][usedefault, addprefix=\global, 1=, 2=]{\mathrm{Hom}_{#1}^{#2}(#3,#4)}

\global\long\def\Var{\mathrm{Var}}
\global\long\def\Gal{\mathrm{Gal}}
\global\long\def\bRep{\mathrm{Rep}}
\global\long\def\Jac{\mathrm{Jac}}
\global\long\def\Gor{\mathrm{Gor}}
\global\long\def\Ker{\mathrm{Ker}}
\global\long\def\Im{\mathrm{Im}}
\global\long\def\Aut{\mathrm{Aut}}
\global\long\def\sht{\mathrm{sht}}
\global\long\def\cZ{\mathcal{Z}}
\global\long\def\st{\mathrm{st}}
\global\long\def\diag{\mathrm{diag}}
\global\long\def\characteristic#1{\mathrm{char}(#1)}
\global\long\def\tors{\mathrm{tors}}
\global\long\def\sing{\mathrm{sing}}
\global\long\def\red{\mathrm{red}}
\global\long\def\pt{\mathrm{pt}}
 \global\long\def\univ{\mathrm{univ}}
\global\long\def\length{\mathrm{length}\,}
\global\long\def\sm{\mathrm{sm}}
\global\long\def\top{\mathrm{top}}
\global\long\def\rank{\mathrm{rank\,}}
\global\long\def\Mot{\mathrm{Mot}}
\global\long\def\age{\mathrm{age}\,}
\global\long\def\Conj#1{\mathrm{Conj}(#1)}
\global\long\def\Mass#1{\mathrm{Mass}(#1)}
\global\long\def\Inn#1{\mathrm{Inn}(#1)}
\global\long\def\bConj#1{\mathbf{Conj}(#1)}
\global\long\def\Tw#1{\mathrm{Tw}(#1)}

\tableofcontents

\section{Introduction}

The aim of this paper is to present an attempt to generalize the motivic
integration to wild Deligne-Mumford stacks. In relation to the McKay
correspondence, Denef and Loeser \cite{MR1905024} developed the motivic
integration applicable to the quotient map $\AA_{k}^{d}\to\AA_{k}^{d}/G$
for a finite group $G\subset GL_{d}(k)$ with $k$ a field of characteristic
zero. Motivated by this, the author \cite{MR2271984,MR2027195} developed
the motivic integration over Deligne-Mumford stacks under some tameness
condition. In this paper, we will make conjectures on how the theory
will be generalized by dropping the tameness condition. The easiest
wild case where $G$ has order equal to the characteristic has been
already studied in \cite{MR3230848}.

Let $A$ be a complete discrete valuation ring with algebraically
closed residue field and $D:=\Spec A.$ For a Deligne-Mumford stack
$\cX$ of finite type over $D,$ we will define a \emph{twisted arc}
of $\cX$ as a representable $D$-morphism $\cE\to\cX$, where $\cE$
is the quotient stack associated to some Galois cover of $D.$ We
expect that we can develop the motivic integration on the space of
twisted arcs of $\cX$, which we will denote by $\cJ_{\infty}\cX.$
Our main conjecture is the following: 
\begin{conjecture}[Conjecture \ref{conj--main}]
For a proper birational morphism $f:\cY\to\cX$ of pure-dimensional
Deligne-Mumford stacks of finite type over $D,$ we have a natural
map $f_{\infty}:\cJ_{\infty}\cY\to\cJ_{\infty}\cX.$ Then, for a measurable
function $F$ on a subset $C\subset\cJ_{\infty}\cX$, we have the
change of variables formula
\[
\int_{C}\LL^{F+w_{\cX}}d\mu_{\cX}=\int_{f_{\infty}^{-1}(C)}\LL^{F\circ f_{\infty}-\ord\Jac_{f}+w_{\cY}}d\mu_{\cY}.
\]
Here $\ord\Jac_{f}$ is the Jacobian order function of $f$ and $w_{\cX}$
and $w_{\cY}$ are canonically defined weight functions on $\cJ_{\infty}\cX$
and $\cJ_{\infty}\cY$ respectively. 
\end{conjecture}
In this paper, we will try to justify the conjecture. To do so, our
main tool is what we call \emph{untwisting, }a technique reducing
twisted arcs to non-twisted ones. We will study the formula in more
detail when $\cY$ is the quotient stack $[\AA_{D}^{d}/G]$ and $\cX$
is the quotient variety $\AA_{D}^{d}/G$ for some linear action $G\curvearrowright\AA_{D}^{d}$
of a finite group $G.$ In this case, we will make a more explicit
and conjectural expression of the weight function.

As an application, we will reach another conjecture which ties stringy
invariants of quotient singularities with mass formulae for extensions
of local fields. For a local field $K$ with residue field having
$q$ elements, Serre \cite{MR500361} proved a mass formula
\[
\sum_{L}\frac{1}{\sharp\Aut(L/K)}\cdot q^{-d(L)}=q^{1-n},
\]
where $L$ runs over isomorphism classes of totally ramified field
extensions $L/K$ of degree $n$ and $d(L)$ the discriminant exponent
of $L/K.$ Bhargava \cite{MR2354798} proved an analogous mass formula
for all \'{e}tale extensions of fixed degree. Also we may consider the motivic version
of such counts, replacing $q$ with $\LL$ and summation with motivic
integration over the space of extensions. From our conjectural change
of variables formula, we will derive an equality between a stringy
invariant of the associated quotient singularity and some motivic
count of extensions of the local field, the fraction field of $A.$
We may regard this as a version of the McKay correspondence. Kedlaya
\cite{MR2354797} and then Wood \cite{MR2411405} worked on this kind
of counting problems in terms of local Galois representations. Their
works might have a more direct relation with our motivic count.\footnote{Afterwards such a relation
was studied in \cite{Wood-Yasuda-I,Wood-Yasuda-II}.}

Another possible application concerns resolution of singularities.
Stringy invariants have a lot of information on resolution of singularities.
Our change of variables formula would be helpful in reducing stringy
invariants for varieties with quotient singularities to those for
smooth Deligne-Mumford stacks, which might be easier to compute. An
existence of resolution impose some constraints on stringy invariants.
Thus, if we find a singular variety with stringy invariant violating
one of them, then we can prove the non-existence of resolution of
singularities.\footnote{Afterwards this viewpoint was adopted in \cite{Wood-Yasuda-II}.} Similar arguments can apply to the problem on the non-existence
of crepant resolution. 

We will also discuss the problem when a family of quotient singularities
has a constant stringy invariant, especially in the case where the
family contains both tame and wild ones. 

There have been considerable developments  \cite{Wood-Yasuda-I,Wood-Yasuda-II,wild-p-adic,Yasuda:2014fk2} on the subject  after
the acceptance of the paper for publication until the final proof. 
I  add several footnotes to mention them. All footnotes in the paper are 
added at the final proof.

\subsection{Acknowledgements}

The author wish to thank Tomoyoshi Ibukiyama and Seidai Yasuda for
letting me know the relevance of Serre's result \cite{MR500361} in
this work. The author also like to thank Shuji Saito for stimulating
discussion on quotient singularities, Kiran S. Kedlaya, Julien Sebag
and Melanie Matchett Wood for reading a draft of the paper and giving
me useful comments, and Fabio Tonini for kindly explaining his result
on the moduli of ramified covers.

\subsection{Conventions}

We fix an algebraically closed field $k,$ a complete discrete valuation
ring $A$ with residue field $k.$ We denote the fraction field of
$A$ by $K$. We put $D:=\Spec A$ and $D^{*}:=\Spec K$: we call
them a \emph{formal disk }and \emph{punctured formal disk }respectively.\emph{
}Always $G$ will denote a finite group. 

As many of our statements will be conjectural, we will often make
heuristic arguments rather than rigorous ones. For this reason, we
often identify a $k$-scheme with its $k$-point set.

\section{The ring of values}

Motivic measures and integrals usually take values in some extension
of the Grothendieck ring of varieties. There are several versions
of such extensions which are slightly different one another. Throughout
the paper, we fix an appropriate one and denote it by $\hat{\cM}=\hat{\cM}_{k}.$
The following are some of its properties which will be necessary in
subsequent sections: 
\begin{enumerate}
\item \label{enu:gen1}$\hat{\cM}$ has an element associated to each scheme
$X$ of finite type over $k$, denoted $[X].$ We write $\LL:=[\AA_{k}^{1}].$ 
\item \label{enu:gen2}$\hat{\cM}$ contains the fractional powers $\LL^{a},$
$a\in\QQ$ of $\LL.$ 
\item \label{enu:rel1}For a bijective morphism $Y\to X$, $[Y]=[X].$ In
particular, $[X]=[X_{\red}]$. Also, if $Y\subset X$ is a closed
subset, then $[X]=[Y]+[X\setminus Y].$ This enables us to determine
$[C]\in\hat{\cM}$ for a constructible subset $C$ of a variety. 
\item \label{enu:rel2}Let $f:Y\to X$ be a morphism such that every fiber
$f^{-1}(x)$ admits a bijective morphism from or to the quotient variety
$\AA_{k}^{n}/H$ for some fixed $n$ and for some linear action $H\curvearrowright\AA_{k}^{n}$
of a finite group $H.$ (We will call $f$ an\emph{ $\LL^{n}$-fibration.})
Then $[Y]=[X]\LL^{n}.$
\item \label{enu:comp}For $r\in\ZZ_{>0}$ and a countable index set $I,$
an infinite sum of the form 
\[
\sum_{i\in I}[X_{i}]\LL^{a_{i}},\,a_{i}\in\frac{1}{r}\ZZ
\]
converges if and only if for every $a\in\RR,$ there exists at most
finitely many $i$ such that $\dim X_{i}+a_{i}\ge a.$
\item \label{enu:realization}There exists a ring homomorphism 
\[
P:\hat{\cM}\to\bigcup_{r=1}^{\infty}\ZZ((T^{-1/r}))
\]
called the \emph{virtual Poincar\'{e} realization}. For a variety
$X,$ $P([X])$ equals the virtual Poincar\'{e} polynomial $P(X)$
of $X$. If $X$ is smooth and proper, then $P(X)=\sum_{i}(-1)^{i}b_{i}(X)T^{i},$
where $b_{i}(X)$ is the $i$-th Betti number for $l$-adic cohomology.
Also we have $P(\LL)=T^{2}.$
\end{enumerate}
Concerning the construction of $\hat{\cM}$, Properties (\ref{enu:gen1})
and (\ref{enu:gen2}) determine generators, (\ref{enu:rel1}) and
(\ref{enu:rel2}) relations and (\ref{enu:comp}) how to complete
the ring. Property (\ref{enu:realization}) is rather a consequence
of the construction. For instance, see \cite{MR3230848} for details.

\section{Motivic integration over varieties}

In this section, we will briefly recall the motivic integration invented
by Kontsevich, developed by Denef and Loeser \cite{MR1664700,MR1905024}
in characteristic zero, and Sebag \cite{MR2075915} (see also \cite{0912.4887})
in positive or mixed characteristic.

\subsection{Varieties over $k$}

Let $X$ be a variety of pure dimension $d$ over $k$. For a non-negative
integer $n,$ the $n$-\emph{jet scheme} of $X$, denoted $J_{n}X$,
is the fine moduli scheme parameterizing $n$-jets: 
\[
J_{n}X=\{\Spec k[t]/(t^{n+1})\to X\}.
\]
They are schemes of finite type over $k$. For $n\ge m,$ we have
a natural morphism $J_{n}X\to J_{m}X$, called the \emph{truncation}.
The \emph{arc space} of $X$, denoted $J_{\infty}X,$ is the projective
limit of $J_{n}X,$ $n\in\ZZ_{\ge0}$, which parameterizes \emph{arcs}:
\[
J_{\infty}X=\{\Spec k[[t]]\to X\}.
\]
 The natural maps $\pi_{n}:J_{\infty}X\to J_{n}X$ are also called
\emph{truncations}. 

The arc space has a measure taking values in $\hat{\cM},$ which is
called the \emph{motivic measure}. We denote it by $\mu_{X}$ and
define it as follows: a subset $C\subset J_{\infty}X$ is said to
be \emph{stable} if for some $n,$ 
\begin{enumerate}
\item $\pi_{n}(C)$ is a constructible subset,
\item $C=\pi_{n}^{-1}(\pi_{n}(C)),$ and
\item for every $n'\ge n,$ $\pi_{n'+1}(C)\to\pi_{n}(C)$ is an $\LL^{d}$-fibration. 
\end{enumerate}
For a stable subset $C,$ we put 
\[
\mu_{X}(C):=[\pi_{n}(C)]\LL^{-nd}\in\hat{\cM},\,(n\gg0).
\]
This defines the motivic measure on stable subsets. We can extend
this to a larger class of subsets called \emph{measurable subsets}.
Roughly, a measurable subset is a subset of $J_{\infty}X$ which can
be approximated by a series of stable subsets. Let 
\[
F:J_{\infty}X\supset C\to\frac{1}{r}\ZZ\cup\{\infty\}
\]
be a measurable function on a subset $C$ for some $r\in\ZZ_{>0}$,
that is, every fiber of it is a measurable subset and $F^{-1}(\infty)$
has measure zero. Then we define the \emph{motivic integral} of $\LL^{F}$
by
\[
\int_{C}\LL^{F}d\mu_{X}:=\sum_{n\in\frac{1}{r}\ZZ}\mu_{X}(F^{-1}(n))\LL^{n}\in\hat{\cM}\cup\{\infty\}.
\]
Here, if the infinite sum diverges, then we put the integral to be
$\infty$.

\subsection{Varieties over $D$}

All these about the motivic integration over $k$-varieties can be
generalized to $D$-varieties. Let $X$ be a $D$-variety, that is,
an integral scheme of finite type over $D.$ For $n\in\ZZ_{\ge0},$
we put $A_{n}:=A/\fm^{n+1},$ $D_{n}:=\Spec A_{n}$. In this setting,
we put 
\[
J_{n}X=\{D\text{-morphisms }D_{n}\to X\},
\]
which is a scheme of finite type over $k$ (not $D$). (The functor
represented by this scheme is called a \emph{Greenberg functor}.)
Then $J_{\infty}X$ is again the projective limit of $J_{n}X$, $n\ge0$
and 
\[
J_{\infty}X=\{D\text{-morphisms }D\to X\}.
\]
Let us suppose also that $X$ is flat and of relative dimension $d$
over $D$. Then, in the same way as in the case of $k$-varieties,
we can define the motivic measure on $J_{\infty}X$ and the motivic
integral $\int_{C}\LL^{F}d\mu_{X}$ for a measurable function $F$
on a subset $C$ of $J_{\infty}X$. The motivic integration for a
$k$-variety $X$ is equivalent to the one for the induced $D$-variety
$X\times_{k}D$ with $D=\Spec k[[t]].$

\section{Motivic integration over Deligne-Mumford stacks}

\subsection{Spaces of $G$-covers of a formal disk}

We mean by a\emph{ $G$-cover} of $D^{*}$ an \'{e}tale $G$-torsor $E^{*}\to D^{*}$.
Then a \emph{$G$-cover} of $D$ is the finite cover $E\to D$ associated
to a $G$-cover $E^{*}\to D^{*}$ of $D^{*}.$ Here $E$ is the unique
normal scheme finite over $D$ containing $E^{*}$ as an open dense
subscheme.
Let $\overline K$ be the algebraic closure of $K$.  A $G$-cover $E^{(*)}\to D^{(*)}$  is called
\emph{pointed }if a lift of the natural $\overline K$-point of $D^{(*)}$ to  $E^{(*)}$
is prescribed. 

\begin{conjecture}
\label{conj: G-cov-moduli}There exist the moduli spaces parameterizing
isomorphism classes of the pointed/unpointed $G$-covers of $D$:
we will denote them by $\GCovPt D$ and $\GCov D$ respectively.
\end{conjecture}
The group $G$ acts on $\GCovPt D$ by changing the given pointing
 or equivalently by replacing the given $G$-action on $E$ with
conjugation. Then we should have 
\[
\GCovPt D/G=\GCov D.
\]
 The conjecture holds if $\characteristic k\nmid\sharp G.$ Indeed
the moduli spaces are just finite points in this case. Moreover $\GCov D\cong\GCov{D^{*}}$
has exactly $\sharp\Conj G$ points. When $A$ has equal characteristic
$p>0$ and $G$ is a $p$-group, Harbater \cite{MR579791} constructed
the (coarse) moduli spaces of $G$-covers of $D^{*}$. If $G_{K}$ denotes
the absolute Galois group of $K,$ then there exists a one-to-one
correspondence between $\GCovPt D$ and the set of continuous homomorphisms
$G_{K}\to G.$ Actually the conjecture above seems to be one of the technical
keys of the whole story in this paper. For instance, the existence
of moduli spaces of $G$-jets/arcs and twisted jets/arcs discussed
below would easily follow from the conjecture along the same line
as in \cite{MR3230848}. A recent result by Tonini \cite{Tonini:2013fk}
may be helpful in addressing this problem. 

Besides Conjecture \ref{conj: G-cov-moduli}, we also expect that
$\GCov D$ is the inductive limit of some series 
\[
V_{0}\to V_{1}\to V_{2}\to\cdots
\]
such that $V_{i}$ are schemes of finite type and $V_{i}\to V_{i+1}$
are injective morphisms. (We expect them to be not immersions but
rather something like immersions followed by Frobenius morphisms,
as such a phenomenon appears in \cite{MR579791}.) We will say
that a subset of $\GCov D$ is \emph{constructible} if it is the image
of a constructible subset of some $V_{i}$. To a constructible subset
$C$ of $\GCov D,$ we associate $[C]\in\hat{\cM}$ in the obvious
way. We obtain the \emph{tautological motivic measure }on $\GCov D$:
\begin{align*}
\tau:\{\text{constructible subsets of }\GCov D\}&\to\hat{\cM}\\
C&\mapsto \tau(C)=[C]
\end{align*}

\subsection{$G$-arcs and jets}

Let $M$ be a $D$-variety endowed with a $G$-action.
\begin{defn}
We define a\emph{ $G$-arc} of $M$ as a $G$-equivariant $D$-morphism
$E\to M$ for some $G$-cover $E\to D$. Two $G$-arcs $E,E'\to M$
are \emph{isomorphic} if there exists an isomorphism $E\to E'$ of
$G$-covers compatible with the given morphisms $E,E'\to M$. 
\end{defn}
For a connected Galois cover $E\to D$ of degree $e,$ we define a
closed subscheme $E_{n}$ of $E$ to be the one having length $1+ne.$
In particular, $E_{0}\cong\Spec k.$ For a non-connected Galois cover
$E=\bigsqcup_{i}E^{i}\to D$ with $E^{i}$ connected components, we
put $E_{n}:=\bigsqcup_{i}E_{n}^{i}.$ 
\begin{defn}
We define a \emph{$G$-$n$-jet} of $M$ as a pair $(E,E_{n}\to M)$
of a $G$-cover $E\to D$ and a $G$-equivariant $D$-morphism $E_{n}\to M.$
Two $G$-$n$-jets $(E,E_{n}\to M)$ and $(E',E_{n}'\to M)$ are \emph{isomorphic}
if there exists an isomorphism $E\to E'$ of $G$-covers such that
the induced isomorphism $E_{n}\to E_{n}'$ is compatible with the
given morphisms to $M.$ \end{defn}
\begin{conjecture}
There exist moduli spaces parameterizing isomorphism classes of $G$-arcs
and $G$-$n$-jets of $M$ for every $n$: we will denote them by
$J_{\infty}^{G}M$ and $J_{n}^{G}M$ respectively. 
\end{conjecture}
If this is the case, we would have natural morphisms 
\[
J_{\infty}^{G}M\to\cdots\to J_{n+1}^{G}M\to J_{n}^{G}M\to\cdots\to J_{0}^{G}M\to\GCov D.
\]
For $E\in\GCov D$ and for each $n,$ we put $J_{n}^{G,E}M\subset J_{n}^{G}M$
to be the fiber over $E.$ 
\begin{lem}
Let $E\in\GCov D$ and $G'\subset G$ the stabilizer of a connected
component $E'$ of $E$, that is, $G':=\{g\in G\mid g(E')=E'\}.$
Let $M_{0}:=M\times_{D} D_0$. (Recall $D_0 =\Spec k$.) Then 
\[
J_{0}^{G,E}M\cong M_{0}^{G'}/N_{G}(G').
\]
Here $N_{G}(G')$ is the normalizer of $G'$ in $G.$ \end{lem}
\begin{proof}
A pointed $G$-0-jet $\alpha:E_{0}\to M_{0}$ is uniquely determined
by $\alpha(E'_{0})\in M_{0}$, which is invariant under the $G'$-action.
Changing the connected component, $\alpha(E_{0}')$ will be changed
to a point of $M_{0}$ invariant under a subgroup conjugate to $G'.$
Thus we have
\[
J_{0}^{G,E}M\cong\left(\bigcup_{H:\text{ conjugate to \ensuremath{G'}}}M_{0}^{H}\right)/G\cong M_{0}^{G'}/N_{G}(G').
\]

\end{proof}

\subsection{Twisted arcs and jets}

Let $\cX$ be a DM stack of finite type over $D.$ 
\begin{defn}
A \emph{stacky formal disk} is a connected and normal DM stack $\cE$
which is birational and finite over $D.$ A \emph{twisted arc }of
$\cX$ is a representable $D$-morphism $\cE\to\cX$ with $\cE$ a
stacky formal disk. (Note that $\cE\to\cX$ is representable if and
only if the induced map of stabilizers of closed points is injective.)
Two twisted arcs $\cE\to\cX$ and $\cE'\to\cX$ are \emph{isomorphic
}if there exists an isomorphism $\cE\to\cE'$ making the diagram
\[
\xymatrix{\cE\ar[r]\ar[dr] & \cE'\ar[d]\\
 & \cX
}
\]
2-commutative. \end{defn}
\begin{conjecture}
There exists a moduli space parameterizing isomorphism classes of
twisted arcs of $\cX:$ we will denote it by $\cJ_{\infty}\cX.$ 
\end{conjecture}
Let us suppose that this conjecture holds.
\begin{lem}
Suppose that $\cX$ is a quotient stack $[M/G]$ for $M$ and $G$
as above. Then we have
\[
\cJ_{\infty}\cX\cong J_{\infty}^{G}M.
\]
\end{lem}
\begin{proof}
Let $E\to M$ be a $G$-arc. Taking the quotient stacks, we obtain
a twisted arc 
\[
[E/G]\to\cX.
\]
This defines a map $J_{\infty}^{G}M\to\cJ_{\infty}\cX.$ Conversely,
for a twisted arc $\cE\to\cX,$ we put $E:=M\times_{\cX}\cE.$ Then
the map $E\to M$ is a $G$-arc. 
\end{proof}
Although not defining \emph{twisted jets, }we expect:
\begin{conjecture}
For each $n\in\ZZ_{\ge0},$ there exists a moduli space $\cJ_{n}\cX$
of twisted $n$-jets of $\cX$ satisfying: 
\begin{enumerate}
\item If $\cX=[M/G],$ then $\cJ_{n}\cX=J_{n}^{G}M.$
\item There exist truncation maps $\cJ_{n+1}\cX\to\cJ_{n}\cX$ making $\cJ_{\infty}\cX$
the projective limit of $\cJ_{n}\cX,$ $n\in\ZZ_{\ge0}.$
\end{enumerate}
\end{conjecture}
Assuming these conjectures, when $\cX$ is of pure dimension $d$
(relatively over $D$), we can define the motivic measure $\mu_{\cX}$
on $\cJ_{\infty}\cX$ in a similar way as the one on the arc space
of a variety. The untwisting technique (Section \ref{sec:Untwisting:-A-justification})
would provide an evidence that there are enough stable or measurable
subsets of $\cJ_{\infty}\cX$.

\section{Formulating the change of variables formula for DM stacks}

\subsection{Maps of twisted arc spaces}

Given a morphism $f:\cY\to\cX$ of DM stacks of finite type over $D$,
we can construct a map

\[
f_{\infty}:\cJ_{\infty}\cY\to\cJ_{\infty}\cX
\]
as follows: let $\gamma:[E/G]\to\cY$ be a twisted arc with $E$ a
connected $G$-cover of $D$. Let $y\in\cY$ be the image of the closed
point by $\gamma$ and $x\in\cX$ the image of $y$ by $f.$ Set 
\[
N:=\Ker(G\xrightarrow{\gamma}\Aut(y)\xrightarrow{f}\Aut(x)).
\]
Then $f\circ\gamma$ factors as
\[
f\circ\gamma:\cE=[E/G]\to\cE'=[(E/N)/(G/N)]\to\cX.
\]
Here $[(E/N)/(G/N)]$ is the quotient stack associated to the induced action of the quotient group $G/N$
on the quotient scheme $E/N$.
We define $f_{\infty}(\gamma)$ to be the induced morphism $\cE'\to\cX$. 
\begin{lem}[Almost bijectivity lemma]
Let $f:\cY\to\cX$ be a morphism of DM stacks of finite type over
$D.$ Suppose that for closed substacks $\cY'\subset\cY$ and $\cX'\subset\cX,$
$f$ induces an isomorphism $\cY\setminus\cY'\xrightarrow{\sim}\cX\setminus\cX'.$
Then $f_{\infty}$ restricts to a bijection
\[
\cJ_{\infty}\cY\setminus\cJ_{\infty}\cY'\to\cJ_{\infty}\cX\setminus\cJ_{\infty}\cX'.
\]
\end{lem}
\begin{proof}
Let $(\gamma:\cE\to\cX)\in\cJ_{\infty}\cX\setminus\cJ_{\infty}\cX'.$
Then we put $\tilde{\cE}$ to be the normalization of the irreducible
component of $\cE\times_{\cX}\cY$ dominating $\cE.$ The induced
morphism $\tilde{\gamma}:\tilde{\cE}\to\cY$ is the unique twisted
arc with $f_{\infty}(\tilde{\gamma})=\gamma.$ This shows the lemma. 
\end{proof}
If $\cY$ and $\cX$ have pure dimension $d$, and if $\dim\cX',\,\dim\cY'<d$,
then we expect that $\cJ_{\infty}\cX'$ and $\cJ_{\infty}\cY'$ have
measure zero as subsets of $\cJ_{\infty}\cX$ and $\cJ_{\infty}\cY$
respectively. Then the lemma says that $f_{\infty}:\cJ_{\infty}\cY\to\cJ_{\infty}\cX$
is almost bijective. Thanks to this, we can expect that motivic integrals
on $\cJ_{\infty}\cX$ are transformed into ones on $\cJ_{\infty}\cY$
and vice versa. The formula describing the transform will be called
the \emph{change of variables formula}.

\subsection{Order functions associated to submodules}

To formulate the change of variables formula, we need to introduce
order functions on the space of twisted arcs.
\begin{defn}
Let $\cX$ be a DM stack of pure dimension $d$, let $N\subset M$
be coherent $ $$\cO_{\cX}$-modules which are generically of rank
one. (In practice, we often take $\cO_{\cX}$, $\Omega_{\cX/D}^{d}$
or $\omega_{\cX/D}$ as $M.$) Then we define the \emph{order function}
of the submodule $N$ 
\[
\ord N:\cJ_{\infty}\cX\to\QQ_{\ge0}\cup\{\infty\}
\]
as follows: let $\gamma:[E/G]\to\cX$ be a twisted arc with $E$ a
connected $G$-cover of $D.$ If $N_{E}$ and $M_{E}$ denote the
pull-backs of $N$ and $M$ to $E,$ and if $M_{E}/\tors$ is free
of rank one, then 
\[
\Im(N_{E}\to M_{E}/\tors)=s^{n}(M_{E}/\tors)
\]
for some $n\in\ZZ_{\ge0}$ with $s$ a uniformizer of the function field of $E$. Then we define
\[
(\ord N)(\gamma):=\frac{n}{\sharp G}.
\]
If $M_{E}/\tors$ is not free of rank one, then we put $(\ord N)(\gamma)=\infty$,
expecting that this does not happen for almost all $\gamma.$ \end{defn}
\begin{rem}
Generalizing this, we can also define the order function of a \emph{fractional}
submodule $N$ of $M,$ that is, an $\cO_{\cX}$-submodule of $M\otimes_{\cO_{\cX}}Q$
with $Q$ the sheaf of total quotient rings. Then the function may
take negative values. \end{rem}
\begin{defn}
For a generically \'{e}tale morphism $f:\cY\to\cX$ of DM stacks,
we define its \emph{Jacobian order function, $\ord\Jac_{f},$ }as
the order function of the submodule $\Im(f^{*}\Omega_{\cX/D}^{d}\to\Omega_{\cY/D}^{d})\subset\Omega_{\cY/D}^{d}.$
\end{defn}

\subsection{The change of variables formula}

The following is the main conjecture of this paper:
\begin{conjecture}[Main Conjecture]
\label{conj--main}For each DM stack $\cX$ of finite type and pure
dimension over $D$, there exists a canonically defined function
\[
w_{\cX}:\cJ_{\infty}\cX\to\left(\bigcup_{x\in\cX}\frac{1}{\sharp\Aut(x)}\ZZ\right)
\]
such that if $\cX$ is a scheme, then $w_{\cX}\equiv0.$ (We call
this the \emph{weight function on $\cJ_{\infty}\cX$.}) Moreover,
for a proper birational morphism $f:\cY\to\cX$ of such stacks, and
for a measurable function $F:\cJ_{\infty}\cX\supset C\to\frac{1}{r}\ZZ\cup\{\infty\},$
we have
\[
\int_{C}\LL^{F+w_{\cX}}d\mu_{\cX}=\int_{f_{\infty}^{-1}(C)}\LL^{F\circ f_{\infty}-\ord\Jac_{f}+w_{\cY}}d\mu_{\cY}.
\]

\end{conjecture}
This is basically of the same form as the formula in the tame setting
proved in \cite{MR2271984}. However differences lie for instance
in the facts that in the wild case, the moduli spaces of twisted arcs
are much larger and that weight functions may take negative and unbounded
values. The conjecture has been proved in \cite{MR3230848} for
the morphism from the quotient stack $[\AA_{k}^{d}/G]$ to the quotient
variety $\AA_{k}^{d}/G$ associated to a linear action of $G=\ZZ/p\ZZ$
with $p=\characteristic k$.

\section{More details on the linear case}

Let $G\subset GL_{d}(A)$ be a finite subgroup, which acts on $V:=\AA_{D}^{d}$:
we call $V$ a $G$\emph{-representation} over $D$. Then we put $\cX:=[V/G]$
and $X=V/G,$ the quotient stack and variety. In this section, we
will study the detailed structure of twisted arc and jet spaces and
obtain a conjectural ``explicit'' expression of the weight function
$w_{\text{\ensuremath{\cX}}}$.

\subsection{Spaces of arcs and jets}

Let 
\[
E=\Spec B\to D=\Spec A
\]
be a $G$-cover. We will study the structure of $J_{\infty}^{G,E}V$
and $\pi_{n}(J_{\infty}^{G,E}V),$ $n\in\ZZ_{\ge0},$ where $\pi_{n}$
are the truncation maps. Let $A[x_{1},\dots,x_{d}]$ be the coordinate
ring of $V.$ We have the induced $A$-linear action of $G$ on $A^{d}=\bigoplus_{i=1}^{d}A\cdot x_{i}$.
This uniquely extends to a $B$-linear action on $B^{d}$. Let 
\[
\rho:G\to\Aut_{B}(B^{d})
\]
be the corresponding map. On the other hand, $G$ acts on $B$ and
diagonally on $B^{d}$, which induces a map 
\[
\delta:G\to\Aut_{A}(B^{d}).
\]
Since $\rho(g),$ $g\in G$ are represented by $d\times d$ matrices
with entries from $A$, for every $g,g'\in G,$ $\rho(g)$ and $\delta(g')$
are commutative. Hence the map
\[
\rho\delta^{-1}:G\to\Aut_{A}(B^{d}),\,g\mapsto\rho(g)\circ\delta(g)^{-1}
\]
is a group homomorphism. 
\begin{defn}
We define $\Xi_{B/A}^{V}\subset B^{d}$ to be the $G$-invariant subset
for $\rho\delta^{-1},$ which is the locus where the two $G$-actions
$\rho$ and $\delta$ coincide. For $n\in\ZZ_{\ge0},$ writing $E_{n}=\Spec B_{n},$
we put $\Xi_{B/A,n}^{V}$ to be the image of $\Xi_{B/A}^{V}$ in $B_{n}^{d}.$\end{defn}

By definition, these modules  $\Xi_{B/A}^{V}$, $\Xi_{B/A,n}^{V}$ have the unique
natural $G$-action. 

\begin{lem}\label{lem:id}
We can identify $J_{\infty}^{G,E}V$ with $\Xi_{B/A}^{V}/G$ and $\pi_{n}(J_{\infty}^{G,E}V)$
with $\Xi_{B/A,n}^{V}/G.$\end{lem}
\begin{proof}
A $G$-arc $E\to V$ corresponds to an equivariant homomorphism
\[
A[x_{1},\dots,x_{d}]\to B.
\]
In turn, this corresponds to an element of $\Xi_{B/A}^{V}.$ 
The set $J_{\infty}^{G,E}V$ is identified with the set of such equivariant homomorphisms
modulo the $G$-action, in turn with  $\Xi_{B/A}^{V}/G$.
The assertion
on jets follows from this.\end{proof}

\begin{prop}
$\Xi_{B/A}^{V}$ is a free $A$-module of rank $d.$ Moreover it is
saturated, that is, if $ax\in\Xi_{B/A}^{V}$ with $a\in A$ and $x\in B^{d}$,
then $x\in\Xi_{B/A}^{V}.$\end{prop}
\begin{proof}
For $a\in A,$ the multiplication map, $B^{d}\to B^{d},$ $x\mapsto ax,$
commutes with all $\rho(g)$ and $\sigma(g).$ This shows the saturatedness.
The freeness follows from the facts that it is torsion-free and that
$A$ is a PID. It remains to show that $ $the rank is $d.$ 

The $G$-action $\rho\delta^{-1}$ makes $B^{d}$ a $G$-equivariant
$B$-module. Let $L$ be the fraction field of $B$ and $p:\Spec L\to\Spec K$
the associated $G$-torsor. Then $\Xi_{B/A}^{V}\otimes_{A}K$ is identified
with $(p_{*}(B^{d}\otimes_{B}L))^{G}.$ We can see that this has rank
$d,$ for instance, by trivializing the $G$-torsor with a base change.
\end{proof}
From the proposition, we can identify $\Xi_{B/A,n}^{V}$ with affine
spaces over $k$, using Witt vectors in the mixed characteristic case
(for instance, see \cite[page 276]{MR1045822}).  Now we can easily
deduce the following:

\begin{cor}
For every $n\in\ZZ_{\ge0},$
 $\pi_{n+1}(\cJ_{\infty}\cX)\to\pi_{n}(\cJ_{\infty}\cX)$ is an $\LL^{d}$-fibration, where
  $\pi_{n}:\cJ_{\infty}\cX\to\cJ_{n}\cX$ is the truncation map.
\end{cor}

\subsection{Weight functions in the linear case}

\begin{defn}
\label{def--weight}Let $\alpha_{1}=(\alpha_{1j})_{1\le j\le d},\dots,\alpha_{d}=(\alpha_{dj})_{1\le j\le d}\in B^{d}$
be an $A$-basis of $\Xi_{B/A}^{V}$ and put $Q:=(\alpha_{ij})\in M_{d}(B).$
Let $G'\subset G$ be the stabilizer of some connected component of
$E.$ Let $V_{0}:=V\times_{D}D_{0}.$ Then we define\footnote{Afterwards it turned out that
this definition is not quite correct; the term $\dim V_0 ^{G'}$ needs to be modified,
although they coincide in some cases; see \cite{wild-p-adic}.} a \emph{weight
function} $w=w_{V}:\GCov D\to\QQ$ by 
\begin{align*}
w(E): & =d-\frac{1}{\sharp G}\length\frac{B}{(\det Q)}-\dim V_{0}^{G'}\\
 & =d-\frac{1}{\sharp G}\length\frac{B^{d}}{B\cdot\Xi_{B/A}^{V}}-\dim V_{0}^{G'}.
\end{align*}
Note that if $B$ is a domain, then the middle terms are written as
$v_{B}(\det Q)/\sharp G$ with $v_{B}$ the normalized valuation of
$B$.\end{defn}
\begin{conjecture}
\label{conj: weight-linear}For the quotient stack $\cX=[V/G]$ associated
to a $G$-representation $V$ over $D,$ the weight function $w_{\cX}$
from Conjecture \ref{conj--main} factors as
\[
w_{\cX}:\cJ_{\infty}\cX=J_{\infty}^{G}M\to\GCov D\xrightarrow{w_{V}}\QQ.
\]

\end{conjecture}
We will show an evidence for the conjecture in the next section.
\begin{example}
Let $\zeta\in k$ be a primitive $m$-th root of 1 with $\characteristic k\nmid m$
and $G:=\left\langle \diag(\zeta^{a_{1}},\dots,\zeta^{a_{d}})\right\rangle \subset GL_{d}(k)$
with $0\le a_{i}<m$ and $\gcd(a_{1},\dots,a_{d})$ prime to $m.$
Then $G$ is cyclic of order $m.$ Let 
\[
E=\Spec k[[t^{1/m}]]\to D=\Spec k[[t]]
\]
be a $G$-cover with $\sigma(t^{1/m})=\zeta t^{1/m}.$ Then for a
natural basis of $\Xi$, the matrix in Definition \ref{def--weight}
becomes 
\[
Q=\diag(t^{a_{1}/m},\dots,t^{a_{d}/m}).
\]
Hence 
\begin{align*}
w(E) & =d-\frac{1}{m}\sum_{i=1}^{d}a_{i}-\sharp\{i\mid a_{i}=0\}\\
 & =d-\frac{1}{m}\sum_{i=1}^{d}a_{i}',
\end{align*}
where $a_{i}'$ is an integer such that $a_{i}'\equiv a_{i}\mod m$
and $0<a_{i}'\le m.$ This agrees with weight functions in \cite{MR1905024,MR2027195,MR2271984}. 
\begin{example}
Suppose that $\characteristic k=p>0$ and $G=\left\langle \sigma\right\rangle $
is cyclic of order $p.$ Also suppose that $G$ acts on $k[x_{1},\dots,x_{d}]$,
$d\le p$ by 
\[
\sigma(x_{i})=\begin{cases}
x_{i}+x_{i+1} & (i<d)\\
x_{d} & (i=d).
\end{cases}
\]
Let 
\[
E=\Spec B\to D=\Spec k[[t]]
\]
be a $G$-cover with ramification jump $j>0$. Namely, if $s\in B$
is a uniformizer, then $v_{B}(\sigma(s)-s)=j+1.$ (Note that $j$
is necessarily prime to $p$.) Then there exist $f_{0},\dots,f_{d-1}\in B$
such that\end{example}
\begin{enumerate}
\item $v_{B}(f_{i})\equiv-ij\mod p$ and $0\le v_{B}(f_{i})<p,$
\item If we put $\delta=\sigma^{*}-\mathrm{id}_{B},$ then 
\[
v_{B}(\delta^{n}(f_{i}))=\begin{cases}
v_{B}(f_{i})+nj & (0\le n\le i)\\
\infty & (n>i),
\end{cases}
\]

\item $(\delta^{i}(f_{0}))_{1\le i\le d},\dots,(\delta^{i}(f_{d-1}))_{1\le i\le d}\in B^{d}$
form a basis of $\Xi_{B}.$
\end{enumerate}

The matrix $Q=(\delta^{i}(f_{j}))_{0\le i,j\le d-1}$ corresponding
to the basis is triangular and its determinant is $\prod_{i=0}^{d-1}\delta^{i}(f_{i}).$
Since 
\[
v_{B}(\delta^{i}(f_{i}))=v_{B}(f_{i})+ij=p\cdot\left\lceil \frac{ij}{p}\right\rceil ,
\]
we have
\begin{align*}
w(E) & =d-\sum_{i=0}^{d-1}\left\lceil \frac{ij}{p}\right\rceil -1\\
 & =-\sum_{i=1}^{d-1}\left\lfloor \frac{ij}{p}\right\rfloor .
\end{align*}
This agrees with the weight function in \cite{MR3230848}. 

\end{example}
These examples would suggest that the tame part positively contributes
to the weight and the purely wild part negatively does,\footnote{Afterwards this phenomenon was also confirmed in the case of permutation representations in \cite{Wood-Yasuda-I}.} and that the
weight function measures how tame/wild a $G$-cover $E\to D$ is.
The formula defining $w_{V}$ contains $\det Q$ and it is not clear
how to compute it in general. Therefore we would like to ask:
\begin{problem}
Does $w(E)$ depend only on numerical invariants of $E$ and $V$
as in the above examples? If it is the case, find a formula.\footnote{Afterwards it turned out that $w(E)$ is described by discriminants for permutation representations \cite{Wood-Yasuda-I}
and for little more complicated ones \cite{Yasuda:2014fk2}. However the problem is still open
for the general case.} 
\end{problem}
To attack the above problem, knowing properties of the weight function
would be helpful. For instance, it is natural to expect the following
properties: 
\begin{enumerate}
\item Let $E\to D$ be a $G$-cover, $E'\subset E$ a connected component,
$G'\subset G$ its stabilizer and $V'$ the same scheme as $V$ endowed
with the induced $G'$ action. Then $w_{V}(E)=w_{V'}(E').$
\item For a trivial $G$-cover $E=\bigsqcup_{g\in G}D\to D,$ $w_{V}(E)=0.$
\item For $G$-representations $V$ and $W$ over $D,$ let $G$ act on
$V\times_{D}W$ diagonally. Then $w_{V\times_{D}W}(E)=w_{V}(E)+w_{W}(E).$
\end{enumerate}

\section{\label{sec:Untwisting:-A-justification}Untwisting: a justification
of conjectures}

In this section, we will try to justify Conjectures \ref{conj--main}
and \ref{conj: weight-linear}. Here we will use a technique which
we call \emph{untwisting}. This would reduce the change of variables
formula for stacks to the one for varieties, which has been already
mostly established. Such an argument was already used by Denef and
Loeser \cite{MR1905024}. Our argument will be somehow more involved
than theirs. One of the reasons is that we have to untwist in the
opposite direction. It is inevitable because the weight function may
take negative values in the wild case.\footnote{These two sentences were wrong. This misunderstanding caused the complicated presentation of the untwisting. However
this was the way how the author reached the definition of weight functions and Conjecture \ref{conj:key-lemma} below. In \cite{Yasuda:2014fk2} he revisited this technique in a more intrinsic way.}

We will only consider the proper birational morphism
\[
\phi:\cX=[M/G]\to X=M/G,
\]
associated to a $G$-variety $M,$ as all keys seem to appear already
in this situation. Then, like the change of variables formula in known
cases, Conjecture \ref{conj--main} would basically follow from:
\begin{conjecture}[Key lemma]
\label{conj:key-lemma}Let $\gamma:\cE\to\cX$ be a twisted arc of
$\cX$ which sends the generic point of $\cE$ into the isomorphism
locus of $\phi.$ Let $\phi_{n}:\pi_{n}(\cJ_{\infty}\cX)\to J_{n}X$
be the natural map. Then for $n\gg0,$ 
\[
[\phi_{n}^{-1}(\phi_{n}(\pi_{n}(\gamma)))]=\LL^{\ord\Jac_{\phi}(\gamma)-w_{\cX}}.
\]

\end{conjecture}

\subsection{Fixing a $G$-cover of $D$}
\begin{lem}
Let $\gamma,\gamma'\in J_{\infty}^{G}M$ be $G$-arcs sending the
generic points into the \'{e}tale locus of the quotient map $M\to X$.
If they have the same image in $J_{n}X$ for $n\gg0,$ then they have
the same image in $\GCov D.$ \end{lem}
\begin{proof}
Let $E\to D$ be the $G$-cover associated to $\gamma$. We consider
only the case where $E$ is connected. We may suppose that $M=\Spec S$
and $X=\Spec R$ with $R=S^{G}.$ Let $f\in R$ be such that $\gamma(E^{*})\subset\Spec S_{f}$
and we can write 
\[
S_{f}=R_{f}[x]/(h(x)).
\]
Let $\bar{h}(x)\in K[x]$ to be the image of $h(x)$ by $R_{f}[x]\to K[x]$
derived from $\gamma.$ Then the function field of $E$ is $L=K[x]/(\bar{h}(x)).$
For $n\gg0,$ the induced $n$-jet $f_{n}(\pi_{n}(\gamma))$ of $X$
determines the coefficients of $\bar{h}(x)$ up to sufficiently high
order. Hence it determines $L$ and $E$. 
\end{proof}

\subsection{The linear case}

Now we consider the case where $M=V=\Spec A[x_{1},\dots,x_{d}]$ with
a linear $G$-action. Fix $E=\Spec B\in\GCov D.$ As in Definition
\ref{def--weight}, we take an $A$-basis of $\Xi_{B/A}^{V}$,
\[
\alpha_{1}=(\alpha_{1j})_{1\le j\le d},\dots,\alpha_{d}=(\alpha_{dj})_{1\le j\le d}\in B^{d},
\]
which we now think of row vectors, 
and put 
\[
Q:= \begin{pmatrix}\alpha_{1}\\
\vdots\\
\alpha _{d}
\end{pmatrix}= (\alpha_{ij})_{1\le i,j\le d}\in M_{d}(B),
\]
which is invertible over $B\otimes_{A}K.$ 
\begin{defn}
Let $t$ be a uniformizer of $A$ and let $l\in\ZZ_{\ge0}$ be such
that $t^{l}Q^{-1}\in M_{d}(B).$ Then $t^{l}Q^{-1}$ defines a $B$-linear
transform on $B^{d}\cong\bigoplus B\cdot x_{j}$. Extending it, we
obtain a $B$-algebra endomorphism $u^{*}$ of $B[x_{1},\dots,x_{d}].$
Putting $V_{E}:=\Spec B[x_{1},\dots,x_{d}]=V\times_{D}E,$ we define
a $B$-linear map 
\[
u:V_{E}\to V_{E}
\]
to be the one corresponding to $u^{*}$. We call $u$ an \emph{untwisting
map. }
\end{defn}
Suppose that $G$ acts on $V_{E}=V\times_{D}E$ diagonally. A
$G$-arc $E\to V$ corresponds to a $G$-equivariant $E$-morphism
$E\to V_{E}.$ For $1\le i \le d$, let $\gamma_{i}:E\to V_{E}$ be the $G$-equivariant
$E$-morphism such that 
\begin{align*}
\gamma_{i}^{*}|_{\bigoplus B\cdot x_{j}}:\bigoplus B\cdot x_{j} & \to B\\
\sum a_{j}x_{j} & \mapsto \alpha_i\begin{pmatrix}a_{1}\\
\vdots\\
a_{d}
\end{pmatrix}.
\end{align*}
Then 
\begin{align*}
(u\circ\gamma_{i})^{*}|_{\bigoplus B\cdot x_{j}}:\bigoplus B\cdot x_{j} & \to B\\
\sum a_{j}x_{j} & \mapsto \alpha_i t^{l}Q^{-1}\begin{pmatrix}a_{1}\\
\vdots\\
a_{d}
\end{pmatrix} = t^l a_i.
\end{align*}
For any $\gamma\in J_{\infty}^{G,E}V,$ $\gamma^*|_{\bigoplus B\cdot x_{j}}$ is an $A$-linear combination of $\gamma_i^*|_{\bigoplus B\cdot x_{j}}$, and $(u\circ\gamma)^{*}|_{\bigoplus B\cdot x_{j}}$
is an $A$-linear combination of $(u\circ\gamma_{i})^{*}|_{\bigoplus B\cdot x_{j}}$.
For $0 \le n \le \infty$, we denote the set of $G$-equivariant morphisms $E_n \to V$ by $\dot{J}_n^{G,E}V$, distinguishing it from $J_n^{G,E}V=(\dot{J}_n^{G,E}V)/G$.
We have the identification $\dot J _n ^{G,E} V = \Xi ^V _{B/A,n}$
(see Lemma  \ref{lem:id} and its proof).

\begin{prop}
For $l\le n\le\infty,$ let $J_{n}^{\ge l}V$ be the subset of $n$-jets
$\gamma:D_{n}\to V$ of $V$ with $\gamma^{*}(x_{i})\in\fm_{A_{n}}^{l},$
where $\fm_{A_{n}}$ is the maximal ideal of $A_{n}.$ We have a natural
bijection
\begin{equation}
u_{\infty}:\dot J _{\infty}^{G,E}V\to J_{\infty}^{\ge l}V.\label{eq:bij}
\end{equation}
Moreover, for $l\le n < \infty ,$ the induced map 
\[
u_{n}:\pi_{n}(\dot J_{\infty}^{G,E}V)\to J_{n}^{\ge l}V
\]
is a trivial  $\AA_{k}^{c}$-fibration with $c=\dim V_{0}^{G'}+d(l-1).$\end{prop}
\begin{proof}
The first assertion is now obvious. The maps $u_{n}$ are linear surjections.
The dimensions of their kernels are independent of $n$ and equal
to

\begin{align*}
c & =\dim J_{l}^{G,E}V-\dim J_{l}^{\ge l}V\\
 & =\dim J_{0}^{G,E}V+d(l-1)\\
 & =\dim V_{0}^{G'}+d(l-1).
\end{align*}
Here $G'\subset G$ is the stabilizer of some connected component
of $E.$ 
\end{proof}
Since $u$ is an isomorphism outside the special fiber of $V_{E}\to D,$
there exists a rational map $\psi\circ u^{-1}:V_{E}\dashrightarrow X_{E}$. 
\begin{lem}
The rational map $\psi\circ u^{-1}$ is defined over $D$. \end{lem}
\begin{proof}
Suppose that this is not true. Then there exists $f\in\mathrm{frac}(A[x_{1},\dots,x_{d}]^{G})$
with $f\notin u^{*}(\mathrm{frac}(A[x_{1},\dots,x_{d}]).$ Then, for
general $\gamma\in J_{\infty}^{\le l}V,$ if $\gamma_{E}:E\to V_{E}$
is the morphism induced from $\gamma$ by extending scalars, then
$\gamma_{E}^{*}(f)\notin K.$ This contradicts the fact that $u_{\infty}^{-1}(\gamma):E\to V$
induces an arc $D\to X$.
\end{proof}
Let $\psi:V\to X$ be the quotient map and $\psi'=\psi\circ u^{-1}:V\dashrightarrow X$
the above rational map over $D.$ 

\begin{equation}
\xymatrix{ & V\ar[dl]_{\psi}\ar[dr]^{u\text{ (defined after extending scalars to \ensuremath{B})}}\\
X &  & V\ar@{-->}[ll]^{\psi'}
}
\label{eq:diagram1}
\end{equation}
This diagram induces the commutative diagram of jet spaces:
\begin{equation}
\xymatrix{ & \pi_{n}(\dot J_{\infty}^{G,E}V)\ar[dl]_{\psi_{n}}\ar[dr]^{u_{n}}\\
\pi_{n}(J_{\infty}X) &  & \pi_{n}(J_{\infty}^{\ge l}V)\ar@{-->}[ll]^{\psi'_{n}}
}
\label{eq:diagram2}
\end{equation}
Here $\psi_{n}'$ should be understood as the restriction of a ``correspondence''
\[
\xymatrix{ & \pi_{n}(\dot J^{G,E}_{\infty}W)\ar[dl]\ar[dr]\\
\pi_{n}(J_{\infty}\bar{X}) &  & \pi_{n}(J_{\infty}^{\ge l}\bar{V})
}
\]
where $\bar{X}$ and $\bar{V}$ are compactifications of $X$ and
$V$ respectively, and $W$ is a resolution of indeterminacies of
$\psi':\bar{V}\dasharrow\bar{X}$. 

We now fix a $G$-arc $\gamma\in \dot J_{\infty}^{G,E}E$ not contained
in a bad locus. We denote the relative dimensions of morphisms $\psi_{n}$
and $u_{n}$ at jets derived from $\gamma$ by $\dim\psi_{n}$ and
$\dim u_{n}$ respectively. Also we write the ``relative dimension
of $\psi_{n}'$'' as $\dim\psi'_{n}$. Then for $n\gg0,$
\[
\dim\psi_{n}=\dim\psi_{n}'+\dim u_{n}=\dim\psi_{n}'+c,
\]
with $c$ as above. 
On the other hand, we would have the associativity of Jacobian orders,
\[
\ord\Jac_{\psi'}=\ord\Jac_{\psi}+\ord\Jac_{u^{-1}},
\]
where $\ord\Jac_{?}$ denotes the Jacobian order of $?$ at the arc
derived from $\gamma$. Since $\Jac_{u^{-1}}=\det(t^{-l}Q)=t^{-dl}\cdot\det Q$
and $\det Q\in B,$ we would have
\[
\ord\Jac_{u^{-1}}\equiv\frac{1}{\sharp G}\length\frac{B}{(\det Q)}-dl.
\]
On the other hand, suitably generalizing the change of variables formula
to our rational map, we would have 
\[
\ord\Jac_{\psi'}=\dim\psi'_{n}.
\]
Hence
\begin{align*}
\dim\psi_{n} & =\dim\psi_{n}'+c\\
 & =\ord\Jac_{\psi}-w_{\cX}(\gamma).
\end{align*}
These arguments would justify Conjecture \ref{conj:key-lemma} and
hence Conjectures \ref{conj--main} and \ref{conj: weight-linear}
in this situation.

\subsection{Untwisting: the non-linear or singular case}

Next we consider the case where the $G$-variety $M$ is just an affine
(possibly singular) $D$-variety.\footnote{Afterwards this case was  studied in \cite{Yasuda:2014fk2} in more details.} There exists an equivariant closed
embedding $M\hookrightarrow V=\AA_{D}^d$ with $G$ linearly acting
on $V.$ From $V,$ we construct the same diagram as (\ref{eq:diagram1}).
Let $Y:=M/G$ and $\bar{M}\subset V$ the closure of $(\psi')^{-1}(Y).$
 The bijection
(\ref{eq:bij}) restricts to a bijection 
\[
\dot J_{\infty}^{G,E}M\to J_{\infty}^{\ge l}\bar{M}:=J_{\infty}^{\ge l}V\cap J_{\infty}\bar{M}.
\]
Diagram (\ref{eq:diagram2}) restricts to:

\begin{equation*}
\xymatrix{ & \pi_{n}(\dot J_{\infty}^{G,E}M)\ar[dl]_{\overline{\psi_{n}}}\ar[dr]^{\overline{u_{n}}}\\
\pi_{n}(J_{\infty}Y) &  & \pi_{n}(J_{\infty}^{\ge l}\bar{M})\ar@{-->}[ll]^{\overline{\psi'_{n}}}
}
\end{equation*}
Here the overlines mean restrictions of maps. With similar notation
as above, we would have 
\begin{align*}
\dim\overline{\psi_{n}} & =\dim\overline{\psi_{n}'}+\dim\overline{u_{n}}\\
 & =\ord\Jac_{\overline{\psi}}+\dim\overline{u_{n}}+\ord\Jac_{\overline{u^{-1}}}.
\end{align*}
Therefore the weight of the relevant arc, $w_{[M/G]}(\gamma),$ would
be 
\[
-\dim\overline{u_{n}}-\ord\Jac_{\overline{u^{-1}}}.
\]
Unlike the linear case, the two terms are not probably be constant
even if we fix a $G$-cover $E$. 
\begin{rem}
Another possible (and more direct) approach to the conjectures is
to generalize Looijenga's argument \cite{MR1886763}: we identify
fibers of $\psi_{n}$ with suitable submodules of
\[
\Hom[B]{\tilde{\gamma}^{*}\Omega_{\cY/\cX}}{\fm_{B}^{a}/\fm_{B}^{b}}
\]
and compute their dimensions. 
\end{rem}

\section{Stringy motifs vs.\ motivic masses of $G$-covers}

\subsection{Stringy motifs}

Let $\cX$ be a normal $\QQ$-Gorenstein DM stack over $D.$ Namely for some
$r>0,$ the double dual $\omega_{\cX/D}^{[r]}:=(\omega_{\cX/D}^{\otimes r})^{\vee\vee}$
of $\omega_{\cX/D}^{\otimes r}$ is invertible. Then we define a function
$F_{\cX}$ on $\cJ_{\infty}\cX$ as the order function of the submodule
\[
(\Omega_{\cX/D}^{d})^{\otimes r}/\tors\subset\omega_{\cX/D}^{[r]}
\]
divided by $r.$
\begin{defn}
Let $W\subset\cX\times _D D_0$ be a closed subset and $(\cJ_{\infty}\cX)_{W}$ the
preimage of $W$ by the natural map $\cJ_{\infty}\cX\to\cX\times _D D_0.$ We define
the \emph{stringy motif }of $\cX$ along $W$ as the integral 
\[
M_{\st}(\cX)_{W}:=\int_{(\cJ_{\infty}\cX)_{W}}\LL^{F_{\cX}+w_{\cX}}d\mu_{\cX}\in\hat{\cM}\cup\{\infty\}.
\]
We say that $\cX$ is \emph{stringily log terminal along $W$ }if
$M_{\st}(\cX)_{W}\ne\infty.$ If $W=\cX,$ we just write $M_{\st}(\cX).$ 
For a stack $\cX$ defined over $k$ and a closed subset $W\subset \cX$, 
we define $M_{\st}(\cX)_{W}$ as $M_{\st}(\cX\otimes _k k[[t]])_{W}$
\end{defn}
Since $F_{\cX}$ comes from singularities of $\cX$ and $\omega_{\cX}$
from local actions of stabilizers, roughly $M_{\st}(\cX)_{W}$ measures
how $\cX$ is far from a smooth variety along $W.$ 
\begin{rem}
We may generalize the above definition to log stacks, that is, pairs
of a DM stack and a $\QQ$-divisor on it. We will not enter this issue
in this paper.
\begin{rem}
A $\QQ$-Gorenstein variety $X$ is log terminal in the usual sense
if it is stringily log terminal. The converse holds if the variety
admits a log resolution. $ $ If $\cX$ is tame and smooth, then $\cX$
is stringily log terminal. In general, a smooth DM stack may not be
stringily log terminal (see \cite{MR3230848}). 
\end{rem}
\end{rem}
We can compute stringy invariants using resolution data if a nice
resolution exists. 
\begin{prop}
\label{prop:explicit-formula}Let $f:Y\to X$ be a resolution of a
log terminal variety $X$ over $k$ such that the relative canonical divisor $K_{f}$ is simple normal crossing,
say written as $K_{f}=\sum_{i\in I}a_{i}E_{i}.$ For a subset $J\subset I$,
we set $E_{J}^{\circ}:=\bigcap_{i\in J}E_{i}\setminus\bigcup_{i\notin J}E_{i}.$
Then 
\[
M_{\st}(X)_{W}=\sum_{J\subset I}[E_{J}^{\circ}\cap f^{-1}(W)]\prod_{j\in J}\frac{\LL-1}{\LL^{a_{i}+1}-1}.
\]
In particular, if $f$ is crepant, that is, $K_{f}=0$, then $M_{\st}(X)_{W}=[f^{-1}(W)].$\end{prop}
\begin{rem}
Originally stringy invariants were defined by Batyrev \cite{MR1672108}
using a formula as in Proposition \ref{prop:explicit-formula}. Denef
and Loeser \cite{MR1905024} found a way to express them as motivic
integrals over the arc space of the given singular variety. 
\begin{rem}
If Conjecture \ref{conj--main} holds, then we can show that the stringy
motif is invariant under $K$-equivalences. Moreover we can generalize
it to log stacks.
\end{rem}
\end{rem}

\subsection{Motivic masses}

We conjectured that $\GCov D$ has the tautological motivic measure
$\tau$. To a $G$-representation $V$ over $D,$ we have associated
the weight function $w_{V}$ on $\GCov D.$ 
\begin{defn}
The \emph{motivic mass }of $G$-covers of $D$ with respect to $V$
is
\[
\Mass{G,D,V}:=\int_{\GCov D}\LL^{w_{V}}d\tau 
= \sum _{r\in \QQ} [w_V^{-1}(r)]\LL ^r
\in\hat{\cM}\cup\{\infty\}.
\]
\end{defn}
\begin{rem}
The motivic mass is analogous to weighted counts of extensions of
a local field by Serre \cite{MR500361} and Bhargava \cite{MR2354798}.
Serre proved a mass formula for totally ramified field extensions
of a local field with finite residue field. Bhargava proved a similar
formula by allowing \'{e}tale extensions. Kedlaya's reinterpretation
\cite{MR2354797} of Bhargava's result in terms of Galois representation,
and Wood's subsequent work \cite{MR2411405} seem to be more directly
related to our motivic invariant. \end{rem}
\begin{prop}
\label{prop:stringy-vs-mass}We denote by 0 the origin of $V\times_D D_0$, and
its images in $X:=V/G$ and $\cX:=[V/G]$ as well. Suppose that Conjectures
\ref{conj--main} and \ref{conj: weight-linear} hold. Suppose also
that $G$ has no reflection, equivalently, the quotient map $V\to X$
is \'{e}tale in codimension one. Then 

\begin{equation}
M_{\st}(X)_{0}=M_{\st}(\cX)_{0}=\Mass{G,D,V}.\label{eq:stringy-mass}
\end{equation}
Moreover, if there exists a crepant resolution $f:Y\to X,$ then these
invariants are also equal to $[f^{-1}(0)].$\end{prop}
\begin{proof}
Let $\phi:\cX\to X$ be the natural morphism. Then $F_{X}\circ\phi_{\infty}=\ord\Jac_{\phi}.$
Hence our conjectures show
\[
\int_{(J_{\infty}X)_{0}}\LL^{F_{X}}d\mu_{X}=\int_{(\cJ_{\infty}\cX)_{0}}\LL^{w_{\cX}}d\mu_{\cX}=\int_{\GCov D}\LL^{w_{V}}d\tau,
\]
and the proposition.\end{proof}
\begin{rem}
If $\characteristic k\nmid\sharp G,$ then $\GCov D$ have exactly
$\sharp\Conj G$ points and
\[
M_{\st}(X)_{0}=\Mass{G,D,V}=\sum_{[g]\in\Conj G}\LL^{\age g}.
\]
See (\ref{eq:age}). Hence, if $f:Y\to X$ is a crepant resolution
and if $\chi_{c}(-)$ denotes the topological Euler characteristic
(for compactly supported cohomology), then 
\[
\chi_{c}(f^{-1}(0))=\sharp\Conj G.
\]
Thus we recover a version of the McKay correspondence. The basically
same result was conjectured by Reid and first proved by Batyrev \cite{MR1677693}
in arbitrary dimension (for the historical account, see \cite{MR1886756}).
Our approach is essentially due to Denef and Loeser \cite{MR1905024}
at least in characteristic zero.
\begin{rem}
Without assuming any conjecture, the proposition above has been established
in \cite{MR3230848} when $G=\ZZ/p\ZZ$, $A=k[[t]]$ and $V$
is defined over $k$. Also, a variant of the proposition, regarded
as a point-counting realization, will be proved in a forthcoming paper \cite{Wood-Yasuda-I}\footnote{This paper is now published online.}
 when $G$ is the $n$-th symmetric group $S_{n}$,
$V$ is the direct sum of two copies of the standard representation
and $Y$ is the Hilbert scheme of $n$ points on the affine plane.
In the same paper, it will turn out that this case is closely related
to Bhargava's mass formula for \'{e}tale extensions.
\end{rem}
\end{rem}

\section{Equisingular families and uniformity of motivic masses}

Motivated by uniformity problems of Kedlaya \cite{MR2354797} and
Wood \cite{MR2411405}, we pose the following problem:
\begin{problem}
\label{prob--uniformity-mass}Given a family $V_{S}\to S$ of $G$-representations
over a scheme $S$ with fibers $V_{s}$, $s\in S,$ are $\Mass{G,D,V_{s}}$
independent of $s\in S$? What finite groups admit such a uniform
family of representations?
\end{problem}
Since $\hat{\cM}$ depends on $k,$ to make the problem precise,
we have to take a suitable realization of motivic masses. However,
in all known cases, they are rational functions in $\LL$, and hence
we do not have to worry about realization. We are mainly interested
in the case where $S$ surjects onto $\Spec\ZZ$ and what happens
around points of characteristic dividing $\sharp G$. 

Proposition \ref{prop:stringy-vs-mass} links the problem to: 
\begin{problem}
\label{prob--equisingular}What families of quotient singularities
are equisingular in the sense that they have the ``same'' stringy
motif?
\end{problem}
These problems are wide open. From now on, we will focus on $G$-representations
with $G$ of prime order. In this case, there have been already explicit
formulae for motivic masses and stringy motifs at least in equal characteristic.
We will now recall them. 

Let $G$ be a finite group of prime order $p$ and $V$ a $G$-representation
of dimension $d$ over $k.$ We first suppose that $p\ne\characteristic k.$
Fixing a primitive $p$-th root $\zeta\in k$ of 1, we write each
element $g\in G$ as $\diag(\zeta^{a_{1}},\dots,\zeta^{a_{d}})$,
$0\le a_{i}<r$ for a suitable basis of $V$, and put
\begin{equation}
\age g:=\frac{1}{r}\sum_{i=1}^{d}a_{i}.\label{eq:age}
\end{equation}
Then 
\[
\Mass{G,D,V}=\sum_{g\in G}\LL^{\age g}.
\]
If $V$ has no reflection, then this is equal to $M_{\st}(X)_{0}$
for the corresponding quotient singularity $0\in X.$

Next suppose that $p=\characteristic k.$ Then $V$ is the direct
sum of indecomposable representations of dimensions $d_{1},\dots,d_{l}$
with $\sum_{i}d_{i}=d.$ (Such a decomposition of $V$ is unique up
to permutation, and we have $d_{i}\le p.$) We put 
\[
D_{V}:=\sum_{i=1}^{l}\frac{d_{i}(d_{i}-1)}{2}.
\]
Then $\Mass{G,D,V}\ne\infty$ if and only if $D_{V}\ge p.$ When one
of the two equivalent conditions holds, then $V$ has no reflection
and we have 
\begin{equation}
\Mass{G,D,V}=M_{\st}(X)_{0}=1+\frac{(\LL-1)\left(\sum_{s=1}^{p-1}\LL^{s+w(s)}\right)}{\LL-\LL^{p-D_{V}}},\label{eq:stringy-p-cyclic}
\end{equation}
with 
\[
w(s)=-\sum_{i=1}^{l}\sum_{j=1}^{d_{i}-1}\left\lfloor js/p\right\rfloor .
\]
 The topological Euler characteristic realization of $\Mass{G,D,V}$
equals $p$ in the tame case and 
\[
1+\frac{p-1}{D_{V}-p+1}
\]
in the wild case. This shows the following:
\begin{prop}
If a family $V_{S}\to S\twoheadrightarrow\Spec\ZZ$ is uniform in
the sense of Problem \ref{prob--uniformity-mass}, then $D_{V_{p}}=p$
for every reduction $V_{p}$ of $V_{S}$ to characteristic $p.$ 
\end{prop}
The following example is a partial converse.
\begin{example}
Let $G$ be a finite group of prime order $p$ and $V_{p}$ a $G$-representation
in characteristic $p$. Then for a suitable basis, a generator $\sigma\in G$
is represented by a Jordan normal form $J$ with blocks
\[
J_{i}=\begin{pmatrix}1 & 1\\
 & 1 & 1\\
 &  & \ddots & \ddots\\
 &  &  & 1 & 1\\
 &  &  &  & 1
\end{pmatrix}
\]
of sizes $d_{i}\le p.$ Thus we may suppose that $V_{p}$ is defined
over $\FF_{p}$. Then we lift $J$ to the matrix $\tilde{J}$ over
$R:=\ZZ[x]/(x^{p}-1)$ with blocks
\[
\tilde{J_{i}}=\begin{pmatrix}1 & 1\\
 & x & 1\\
 &  & \ddots & \ddots\\
 &  &  & x^{d_{i}-2} & 1\\
 &  &  &  & x^{d_{i}-1}
\end{pmatrix}.
\]
For every $i,$ the reduction of $\tilde{J}_{i}$ to a characteristic
$\ne p$ has $d_{i}$ distinct eigenvalues all of which are $p$-th
roots of 1. Hence $\tilde{J_{i}}^{p}=1$ and $\tilde{J}^{p}=1.$ Therefore
$\tilde{J}$ induces a $G$-representation $V_{R}$ lifting $V_{p}$
over $R.$ 

Now suppose that $D_{V_{p}}=\sum_{i}\frac{d_{i}(d_{i}-1)}{2}=p.$
Then 
\[
\Mass{G,D,V_{p}}=1+\sum_{s=1}^{p-1}\LL^{s+w_{V_{p}}(s)},
\]
while for a reduction $V_{l}$ of $V_{R}$ to characteristic $l\ne p$,
\[
\Mass{G,D,V_{l}}=1+\sum_{s=1}^{p-1}\LL^{\alpha(s)}.
\]
Here $\alpha(s)=\sum_{i=1}^{l}\sum_{j=1}^{d_{i}}\left\{ js/p\right\} $,
$\left\{ -\right\} $ denoting the fractional part. We now claim that
$\alpha(s)=s+w_{V_{p}}(s)$. Indeed,
\begin{align*}
s & =\frac{s\cdot D_{V_{p}}}{p}\\
 & =\sum_{i=1}^{l}\sum_{j=1}^{d_{i}-1}\frac{sj}{p}\\
 & =\sum_{i}\sum_{j}\left\lfloor \frac{sj}{p}\right\rfloor +\sum_{i}\sum_{j}\left\{ \frac{sj}{p}\right\} \\
 & =-w_{V_{p}}(s)+\alpha(s).
\end{align*}
Hence, thinking of $\LL$ as an indeterminate, the motivic mass $\Mass{G,D,V_{l}}$
is independent of the characteristic $l$. 
\begin{example}
In the preceding example, suppose that $V_{p}=W^{\oplus p}$ with
$W$ the unique indecomposable $G$-representation of dimension $2$
(in characteristic $p$). Then for any characteristic $l,$ 
\[
\Mass{G,D,V_{l}}=\sum_{i=0}^{p-1}\LL^{i}=[\PP^{p-1}].
\]

\end{example}
\end{example}

\section{Non-existence of resolution of singularities}

Stringy motifs of log terminal varieties have a lot of information
on their resolutions of singularities if exist. A detailed knowledge
of this invariant might lead to the non-existence of resolution of
singularities in positive characteristic. Conjecture \ref{conj--main}
enables us to reduce the stringy invariant of a quotient variety to
that of the corresponding quotient stack, which would be easier to
compute in some cases. For this purpose, we have to know what constraints
the existence of resolution imposes on the invariant.

\subsection{Rationality and the Poincar\'{e} duality}
\begin{defn}
For a $\QQ$-Gorenstein $k$-variety $X$ and a closed subset $W\subset X,$
we put 
\[
P_{\st}(X,\Delta)_{W}:=P(M_{\st}(X,\Delta)_{W})\in\bigcup_{r>0}\ZZ((T^{-1/r})).
\]
with $P$ the virtual Poincar\'{e} realization. If $W=X,$ we omit
the subscript $W.$ 
\begin{defn}
A rational function $f(T^{1/r})\in\ZZ(T^{1/r})$ in $T^{1/r}$ is
said to \emph{satisfy the $d$-dimensional Poincar\'{e} duality }if
\[
T^{2d}\cdot f(T^{-1/r})=f(T^{1/r}).
\]

\end{defn}
\end{defn}
\begin{prop}
Let $X$ be a $\QQ$-Gorenstein $k$-variety which is log terminal along
$W.$ Suppose that there exists a resolution $f:Y\to X$ such that
$K_{f}$ is simple normal crossing. Then $P_{\st}(X)_{W}$ is a rational
function. Moreover, if $X$ is proper and $W=X$, then $P_{\st}(X)$
satisfies the $d$-dimensional Poincar\'{e} duality with $d=\dim X.$\end{prop}
\begin{proof}
The rationality follows from Proposition \ref{prop:explicit-formula}.
We can prove the Poincar\'{e} duality in the same way as Batyrev
\cite{MR1672108} did.
\end{proof}

\subsection{Medusa singularities}

Next, let us recall the following result due to Kerz and Saito:
\begin{thm}[\cite{MR3119100}]
Let $M$ be a smooth $k$-variety with a finite group $G$ acting
on it, 
\[
\pi:M\to X=M/G
\]
the quotient map. Let $S\subset X$ be a closed subset containing
the singular locus such that the reduced preimage $T:=\pi^{-1}(X_{\sing})_{\red}$
is smooth. Suppose that a functorial resolution of singularities holds
over $k.$ Let $f:Y\to X$ be a resolution such that $E=f^{-1}(S)$
(with reduced structure) is a simple normal crossing divisor and $f$
is an isomorphism over $X\setminus X_{\sing}.$ (We will call such
$f$ a log resolution of $X$ and $S$.) Then the dual complex $\Gamma(E)$
of $E$ is contractible. 
\end{thm}
Putting this conversely, if either
\begin{enumerate}
\item there exists a log resolution $f:Y\to X$ of $X$ and $S$ with $\Gamma(f^{-1}(S))$
\emph{not }contractible, or
\item $X$ and $S$ do not admit any log resolution $f:Y\to X$ of $X$
and $S$ with $\Gamma(f^{-1}(S))$ contractible, 
\end{enumerate}
then a functorial resolution over $k$ never holds. Saito calls such
singularities of $X$ \emph{Medusa singularities.} Since stringy motifs
have a lot of information on resolution, we might use them to find
Medusa singularities (if exist) and disprove the functorial resolution.
For this purpose, a key result is the following:
\begin{prop}
\label{prop:poincare-weight-0}Let $X$ be a $\QQ$-Gorenstein variety
which is stringily log terminal along a closed subset $S\subset X.$
Let $f:Y\to X$ be a log resolution of $X$ and $S$ such that $K_{f}$
is a simple normal crossing divisor supported in $E:=f^{-1}(S).$
Suppose that there exists a closed subset $R\subset S$ such that
$R$ is proper over $k$ and $F:=f^{-1}(R)$ is a normal crossing
variety having the same dual complex as $E.$ Then 
\[
\chi(\Gamma(E))=\chi(\Gamma(F))=P_{\st}(X)_{R}|_{T=0},
\]
where $\chi$ is the Euler characteristic.\end{prop}
\begin{proof}
The following is a modification of Batyrev's computation in \cite{MR1672108}.
Write $E=\bigcup_{i\in I}E_{i},$ $F=\bigcup_{i\in I}E_{i}$ and $K_{Y/X}=\sum a_{i}E_{i}$.
Since 
\[
P(F_{J}{}^{\circ})=\sum_{J\subset J'}(-1)^{\sharp J-\sharp J'}P(F_{J'}),
\]
we have 
\begin{align*}
P_{\st}(X)_{R} & =\sum_{\emptyset\ne J\subset I}P(F_{J}^{\circ})\prod_{j\in J}\frac{T^{2}-1}{T^{2(a_{i}+1)}-1}\\
 & =\sum_{\emptyset\ne J\subset I}\left(\sum_{J\subset J'}(-1)^{\sharp J-\sharp J'}P(F_{J'})\right)\prod_{j\in J}\frac{T^{2}-1}{T^{2(a_{i}+1)}-1}\\
 & =\sum_{\emptyset\ne J\subset I}P(F_{J})\left(\prod_{j\in J}\left(\frac{T^{2}-1}{T^{2(a_{i}+1)}-1}-1\right)-(-1)^{\sharp J}\right).
\end{align*}
Substituting 0 for $T,$ we obtain 
\[
\sum_{\emptyset\ne J\subset I}(-1)^{\sharp J-1}\sharp\pi_{0}(F_{J}),
\]
where $\pi_{0}$ denotes the set of connected components. This clearly
equals to the Euler characteristic of the dual complex of $F$ and
to that of $E$.\end{proof}
\begin{rem}
Looking at finer realizations, we might able to capture finer invariants
of $\Gamma(E).$ \end{rem}
\begin{example}
Let $G$ be the cyclic group of prime order $p>0$ and $V$ a $G$-representation
over $k$ of characteristic $p>0.$ Suppose that $V$ is isomorphic
to the direct sum of $l$ indecomposable representations and that
$V$ has no reflection. Then the invariant locus $V^{G}\subset V$
and the singular locus $X_{\sing}$ of the quotient variety $X:=V/G$
are isomorphic to $\AA_{k}^{l}.$ The the additive group $\GG_{a}^{l}\cong\AA_{k}^{l}$
acts on $X$ and transitively on $X_{\sing}.$ Therefore, if a functorial
resolution exists over $k,$ then the resolution $Y\to X$ of $X$
is $\GG_{a}^{l}$-equivariant. Put $S=X_{\sing}$ and $R$ to be any
point of $S.$ Suppose that $D_{V}\ge p.$ Then the assumption of
Proposition \ref{prop:poincare-weight-0} holds, and from Equation
\ref{eq:stringy-p-cyclic}, 
\[
P_{\st}(X)_{R}|_{T=0}=1.
\]
Hence we cannot see whether $\Gamma(E)$ is contractible or not in
this case. 
\end{example}

\subsection{Crepant resolutions}
\begin{problem}
What quotient singularity admits a crepant resolution?
\end{problem}
Crepant resolutions are certainly interesting, for instance, since
they give a simple geometric meaning to stringy invariants and motivic
masses. In characteristic zero, Ito \cite{MR1307302,MR1366564}, Markushevich
\cite{MR899851,MR1464903,MR1273078} and Roan \cite{MR1380512,MR1284568}
proved that every three-dimensional Gorenstein quotient singularity
admits a crepant resolution. On the other hand, our knowledge is quite
limited in the wild situation. As far as the author knows, the following
are only known cases: 
\begin{enumerate}
\item Rational double points which are wild quotient singularities in Artin's
classification \cite{MR0450263}. In this case, the minimal resolution
is crepant.
\item $G=S_{n}$ and $V$ is the direct sum of two copies of the standard
representation. Then the Hilbert scheme of $n$ points on the affine
plane is a crepant resolution of the quotient variety, which is isomorphic
to the $n$-th symmetric product of the affine plane (see \cite[page 229]{MR2107324}
for a historical account of this fact).
\item $\characteristic k=\sharp G=3$ and $V=V_{3,}$ the 3-dimensional
indecomposable representation. See \cite{MR3230848}.
\end{enumerate}
We will be able to use stringy invariants also to study the above
problem. For instance, we can show that the second example with $n=2$
and the last one are only possible cases for linear actions of a cyclic
group of prime order in dimension $\le4.$ It follows from the fact
that $D_{V}=p$ is a necessary condition for the existence of crepant
resolution. It is because otherwise the topological Euler characteristic
realization is not an integer, which is not allowed if a crepant resolution
exists. The same reasoning shows that if a crepant resolution exists,
then $P_{\st}(X)$ is a polynomial in $T$ (not in $T^{1/r}$). 

\bibliographystyle{plain}
\bibliography{/Users/Takehiko/Dropbox/Math_Articles/mybib}

\end{document}